\documentclass[reqno]{amsart}
\usepackage{url}
\usepackage{amssymb,amsthm,amsfonts,amstext,amsmath}
\usepackage{newcent}         
\usepackage{helvet}          
\usepackage{courier}         
\usepackage{graphicx}
\usepackage{enumerate}
\usepackage{hyperref}
\numberwithin{equation}{section}
\usepackage{color}
\usepackage{float}
\usepackage{mathrsfs}
\usepackage{tikz}
\usetikzlibrary{backgrounds}
\usetikzlibrary{patterns,fadings}
\usetikzlibrary{arrows,decorations.pathmorphing}
\usetikzlibrary{calc}
\definecolor{light-gray}{gray}{0.95}
\usepackage{float}
\newtheorem{theorem}{Theorem}[section]
\newtheorem{lemma}[theorem]{Lemma}
\newtheorem{proposition}[theorem]{Proposition}
\newtheorem{corollary}[theorem]{Corollary}

\newtheorem{definition}{Definition}
\newcommand{\mc}[1]{{\mathcal #1}}

\newcommand{\bb}[1]{{\mathbb #1}}

\newcommand{\eps}{\varepsilon}
\newcommand{\kk}{{\bf k}_{\textrm{\rm B}}}

\newcommand{\pfrac}[2]{\genfrac{}{}{}{1}{#1}{#2}}

\let\oldtocsection=\tocsection
\let\oldtocsubsection=\tocsubsection
\let\oldtocsubsubsection=\tocsubsubsection
\renewcommand{\tocsection}[2]{\hspace{0em}\oldtocsection{#1}{#2}}
\renewcommand{\tocsubsection}[2]{\hspace{1em}\oldtocsubsection{#1}{#2}}
\renewcommand{\tocsubsubsection}[2]{\hspace{2em}\oldtocsubsubsection{#1}{#2}}
\DeclareRobustCommand{\SkipTocEntry}[5]{}


\begin{document}

\title[KMP model with  slow/fast boundaries]{Phase transition in the KMP model\\ with  slow/fast boundaries}

\author{Tertuliano Franco}
\address{UFBA\\
 Instituto de Matem\'atica, Campus de Ondina, Av. Adhemar de Barros, S/N. CEP 40170-110\\
Salvador, Brazil}
\curraddr{}
\email{tertu@ufba.br}
\thanks{}


\begin{abstract} The Kipnis-Marchioro-Presutti (KMP) is a known model consisting on a one-dimensional chain of mechanically uncoupled oscillators, whose interactions occur via independent Poisson clocks: when a Poisson clock rings, the total energy at two neighbors is   redistributed uniformly at random between them. Moreover, at the boundaries, energy is exchanged with reservoirs of fixed temperatures. We study here a generalization of the KMP model by considering different rates at energy is exchanged with the reservoirs, and we then prove the existence of a  phase transition for the heat flow.

\end{abstract}

\maketitle


\section{Introduction}\label{s1}
How microscopic interactions determine the macroscopic behavior of a giv-en system is a question that guides a vast research in Statistical Mechanics and Probability. In this context, since the seventies, a rigorous mathematical theory has been developed in order to give a precise sense to the limit from microscopical systems with stochastic time evolution towards its macroscopic point of view.
As a classical reference in the subject we cite the book \cite{KipnScal1999}.

A particular important microscopic system  is the Kipnis-Marchioro-Presutti (KMP) model, see \cite{kmp}. Such model consists of a one-dimensional finite chain of oscillators, being each oscillator described by its velocity and position. The oscillators interact in the following way. Associated to each pair of neighbors there is a Poisson clock; when a certain clock rings, the total energy at the pair of neighbors is redistributed between them uniformly at random. The respective new positions and velocities are then chosen uniformly (according to the Lebesgue measure) among all the possible configurations on its surfaces of constant energy. Besides, at the right (resp. left) boundary, also at arrival times of a Poisson clock, the energy is replaced according to an exponential distribution of parameter $\beta_+>0$ (resp. $\beta_->0$). This is equivalent to say that the system is in contact with reservoirs of temperature $T_\pm=1/(\kk \beta_\pm)$, where $\kk$ stands for the Boltzmann constant. 

As explained in \cite{kmp}, at the invariant state, conditionally on the energy, position and velocity are uniformly distributed. We therefore restrict our attention here only to the energy profile.

Supposing $T_+\neq T_-$, a flux of energy is observed in KMP model. This is the content of \cite{kmp}, i.e., a rigorous proof of the Fourier Law.
Due to its peculiar structure, which gives rise to an interesting duality and consequent manageability, the KMP model is an interesting object of study both in Probability and Statistical Mechanics. See for instance  \cite{BGL2005,GKRV} and references therein.

Recently, several works have investigated how a slowed defect can utterly modify the scaling limit of a given microscopic system. See, for instance, \cite{fgn1,fgn2,fgn3,fgschutz2015,fgsimon2015}. By a slowed defect we mean that some specific site (or bond, or boundary), is rescaled differently from the rest of the system. This is precisely what we investigate in this paper. Rescaling  rates at which energy is exchanged with the reservoirs, we arrive at a phase transition for the steady state. In fact, taking rates as  $AL^{-a}$ and $BL^{-b}$, where  $L$ is the scaling parameter, the invariant profile of temperatures disconnects from the reservoirs if some of the parameters $a$ or $b$ is equal to one. This result, which is the so-called \textit{local equilibrium} (see \cite{kl}) is stated in the Theorem~\ref{thm21}, where explicit formulas for the limiting profiles are also provided.

According to the seminal paper \cite{kmp}, the KMP model has the striking property that its dual process of particles remains invariant (in some sense) when particles are added to the system. This is a fundamental ingredient in order to obtain the steady temperature profile. Since the original proof of this fact given in \cite[Proposition 3.1]{kmp} is somewhat unclear, we present here a simple, easily comprehensible proof of this fact in the Proposition \ref{prop33}, based on a  combinatorial identity. Our proof suits for the slow/fast boundaries case, covering the original model \cite{kmp} as a particular case.

Since the KMP model is non-gradient (see \cite{KipnScal1999} for a precise definition), the hydrodynamic limit of the model presented here turns to be a challenging problem. In view of Theorem \ref{thm21}, we conjecture that its hydrodynamic behavior should be described by a non-linear heat equation with  boundary conditions related to the regimes described in Theorem \ref{thm21}.

The outline of the paper is: Section \ref{s2} presents statements. Section \ref{s3} deals with the duality of KMP process with slow/fast boundaries. In Section \ref{s4}, the Label Process is defined,  which allows explicit computations, eventually leading to the proof of Theorem \ref{thm21}. In Section \ref{sec5}, further extensions and open problems are considered.

\section{Statement of results}\label{s2}
\textbf{Notations:} The cardinality of a finite set $A$ will be denoted by $|A|$. We clarify that here $\bb N=\{0,1,2,\ldots\}$ and $\bb R_+=\{x\in \bb R\,;\, x>0\}$.\medskip
 
For  a positive integer $L$, consider the state space $\Omega_L=\bb R_+^{(2L+1)}$, which represents the set of energy configurations of $2L+1$ oscillators in a one-dimensional chain. We denote an energy configuration by 
\begin{equation}\label{eq21}
\xi\;=\; (\xi_{-L},\ldots,\xi_{L})\in \Omega_L\,.
\end{equation}
 The \textit{KMP model  with slow/fast boundaries} we define here is the  Markov process $\{\xi_t\,;\, t\geq 0\}$  on $\Omega_L$  characterized by its generator $G_L$ acting on smooth bounded functions $f:\Omega_L\to\bb R$ as
\begin{equation*}
\begin{split}
 \big(G_L f\big)(\xi)\!= &\!\! \sum_{x=-L}^{L} \int_0^1\!  \Big[ f(\xi_{-L},\ldots, p(\xi_x+\xi_{x+1}), (1-p)(\xi_x+\xi_{x+1}),\ldots,\xi_L) - f(\xi)\Big]dp\\
 & + \frac{A}{L^a} \int_0^\infty \Big[ f(y,\xi_{-L+1},\ldots,\xi_L) - f(\xi)  \Big] \beta_- e^{-\beta_-y}\,dy\\
  & + \frac{B}{L^b}\int_0^\infty \Big[ f(\xi_{-L},\ldots,\xi_{L-1}, y) - f(\xi)  \Big] \beta_{+} e^{-\beta_+y}\,dy\,,
\end{split}
\end{equation*}
where $\beta_+,\beta_-,A,B>0$ and $a,b\in \bb R$.  We define $T_\pm$, the temperatures in the left and right reservoirs, respectively, by the equalities
$$\beta_{\pm}\;=\;\frac{1}{{\kk}T_{\pm}}\,,$$
where $\kk$ stands for the Boltzmann constant. Some authors assume $\kk=1$ for simplicity.
\begin{definition}
For given $T>0$, let $\nu_T$ be the Gibbs measure (for the energy) of independent oscillators on $\bb Z$. In other words, $\nu_T$ is the following product measure on $\bb R^{\bb Z}_+$: 
\[
d\nu_T\;=\;\prod_{x\in \bb Z} \Big[\frac{1}{{\kk}T}\exp\Big\{-\frac{\xi_x}{{\kk}T}\Big\}\,d\xi_x\Big]\,.
\]
\end{definition}
Notice that the marginal of the measure $\nu_T$ at any site is an exponential distribution of parameter $1/\kk T$.
\begin{definition} Let us divide $\bb R^{\bb Z}_+$ in blocks of size $(2L+1)$, where one of the blocks is centered at the origin. In other words, we write $\bb R^{\bb Z}_+$ as the following Cartesian product:
\[
\bb R^{\bb Z}_+\;=\; \prod_{j \in \bb Z} \bb R^{\{-L,\ldots,0,\ldots, L\}\,+\,j (2L+1)}_+\,.
\]  
Let $\mu_L$ be the (unique) invariant measure of the KMP process with slow/fast reservoirs. Denote by $\tilde{\mu}_L$ its extension to $\bb R^{\bb Z}$ obtained by taking the product of copies of $\mu_L$ on each of the blocks above of size $(2L+1)$.
\end{definition}
The particular choice for the extension of $\mu_L$ has no relevance here and any other extension would suit our purpose as well. We have constructed $\tilde{\mu}_L$  only to give sense to  the statement below.
\begin{theorem}\label{thm21}
Denote by $\mu_L$ the invariant measure of the KMP process with slow/fast boundaries, and let $\tau_{\lfloor uL \rfloor}$ be the shift of $\lfloor uL \rfloor$, with $u\in(-1,1)$.

 Then, as $L\to\infty$, the probability measure $\tau_{\lfloor uL \rfloor}\tilde{\mu}_L$ converges weakly to $\nu_T$,  where:
 \begin{enumerate}[\bf{(}\bf i\bf{)}]
 \item If $a,b<1$, 
 \[
 T\;=\;T(u)=\Big(\frac{1-u}{2}\Big)T_-+\Big(\frac{1+u}{2}\Big)T_+.
 \]
 \item If $a=1$ and $b<1$, 
 \[
 T(u)=\Big(\frac{A(1-u)}{2A+1}\Big)T_-+\Big(\frac{A(1+u)+1}{2A+1}\Big)T_+\,.
 \]
 \item If $a<1$ and $b=1$, 
 \[
 T(u)=\Big(\frac{B(1-u)+1}{2B+1}\Big)T_-+\Big(\frac{B(1+u)}{2B+1}\Big)T_+\,.
 \]
 \item If $a=b=1$,
 \[
 T(u)=\Big(\frac{AB(1-u)+A}{2AB+A+B}\Big)T_-+\Big(\frac{AB(1+u)+B}{2AB+A+B}\Big)T_+\,.
 \]
  \item If $a=b>1$, 
 \[
 T(u)= \Big( \frac{A}{A+B}\Big)T_-+\Big(\frac{B}{A+B}\Big)T_+\,.
 \]
 \item If $a>\max\{1,b\}$,
  \[
 T(u)= T_+\,.
 \]
  \item If $b>\max\{1,a\}$,
  \[
 T(u)= T_-\,.
 \]
 \end{enumerate}
\end{theorem}
Some remarks: The regime  \textbf{(i)} includes the seminal result of \cite{kmp}, which corresponds to the case $a=b=0$  and $A=B=1$. 
In the regimes \textbf{(ii)}, \textbf{(iii)} and \textbf{(iv)}, the temperature varies linearly for $u\in (-1,1)$, but does not interpolate $T_-$ and $T_+$. In fact, when some of the parameters $a$ and $b$ is equal to one, the temperature close to the boundary does not reach the temperature $T_\pm$  of the corresponding reservoir.  In the regimes \textbf{(v)}, \textbf{(vi)} and \textbf{(vii)}, the temperature on the chain of oscillators is completely homogenized. See the Figure \ref{fig1} for an illustration of the regime as a function of the parameters  $a$ and $b$.

\begin{figure}[H]
\centering
\begin{tikzpicture}

\begin{scope}[xshift=2.5cm]
\fill[fill=black!15!white] (-1,-1)--(1,-1)--(1,1)--(-1,1)--cycle;
\fill[pattern=dots] (1,-1.05)--(2.55,-1.05)--(2.55,2.55)--(1,1)--cycle;
\fill[pattern=north west lines] (-1,1)--(1,1)--(2.5,2.5)--(-1,2.5)--cycle;
\draw[ultra thick, gray] (1,1)--(2.6,2.6);
\draw[ultra thick] (1,1)--(1,-1.2);
\draw[ultra thick, dash pattern=on 3.5pt off 1.5pt] (-1.2,1)--(1,1);
\filldraw[fill=white, draw=black, thick] (1,1) circle (2pt);

\draw[->] (-1.2,0)--(2.8,0) node[below]{$a$};
\draw[->] (0,-1.2)--(0,2.8) node[left]{$b$};
\draw (-0.15,1) node[below]{$1$};
\draw (0.85,0) node[below]{$1$};
\end{scope}

\begin{scope}[yshift=-1.5cm]


\fill[black!15!white] (0,0) rectangle (1,-0.5); 
\draw (1,-0.25) node[right]{Regime \textbf{(i)}};

\draw[ultra thick, dash pattern=on 3.5pt off 1.5pt] (0,-1)--(1,-1);
\draw  (1 ,-1) node[right]{Regime \textbf{(ii)}};

\draw[ultra thick] (0,-1.75)--(1,-1.75);
\draw (1,-1.75) node[right]{Regime \textbf{(iii)}};

\filldraw[fill=white, draw=black, thick] (0.5,-2.5) circle (2pt) ;
\draw (1,-2.5) node[right]{Regime \textbf{(iv)}};

\draw[ultra thick,color=gray] (4,-0.25)--(5,-0.25);
\draw (5,-0.25) node[right, color=black]{Regime \textbf{(v)}};

\fill[pattern=dots] (4,-0.75) rectangle (5,-1.25); 
\draw (5,-1) node[right]{Regime \textbf{(vi)}};

\fill[pattern=north west lines] (4,-1.5) rectangle (5,-2); 
\draw (5,-1.75) node[right]{Regime \textbf{(vii)}};
\end{scope}
\end{tikzpicture}
\vspace{0.5cm}
\caption{Regimes for the heat flow}
\label{fig1}
\end{figure}
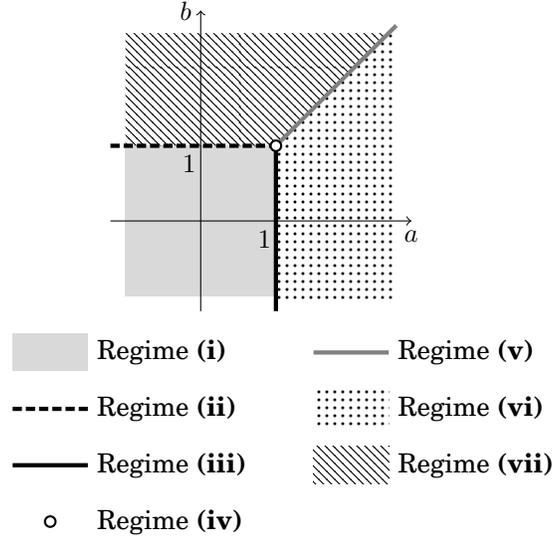

\section{Dual process of walkers}\label{s3}
 We construct in this section a discrete system of walkers which is dual (in a sense to be defined)  to the KMP process with slow/fast boundaries. We adapt here ideas  from \cite{kmp}. 

\begin{definition}\label{def3}
Let $\Lambda_L= \{-L,\ldots,L\}\cup \{ \,\delta(-)\,,\, \delta(+)\,\}$ and denote
\begin{equation}\label{eq31}
n=(n_{\delta(-)},n_{-L},\ldots,n_L,n_{\delta(+)})\in \bb N^{\Lambda_L}
\end{equation}
 Consider the Markov process taking values on $\bb N^{\Lambda_L}$ characterized by the following generator 
\begin{equation*}
\begin{split}
&\big(A_Lf\big)(n)= \frac{A}{L^a}\Big[f(n_{\delta(-)}+n_{-L},0,n_{-L+1},\ldots,n_{\delta(+)})-f(n)\Big]\\
& +\frac{B}{L^b}\Big[f(n_{\delta(-)},\ldots,n_{L-1},0,n_{\delta(+)}+n_L)-f(n)\Big]\\
&
 +\!\sum_{j=-L}^L\!\! \pfrac{1}{n_j+n_{j+1}+1}\!\!\!\!\!\sum_{q=0}^{n_j+n_{j+1}}\!\!\Big[
f(n_{\delta(-)},n_{-L},\ldots,n_{i-1},q,n_i+n_{i+1}-q,\ldots,n_{\delta(+)})-f(n)\Big].\\
\end{split}
\end{equation*}
\end{definition}
This particle system can be described in words as follows. We associate to each pair of neighbour sites  a Poisson clock of parameter one. When a Poisson clock rings, the particles in the corresponding sites are uniformly redistributed. Moreover, associated to the site $-L$ there is a Poisson clock of parameter $A/L^a$. When this Poisson clock rings, all the particles at the site $-L$  move to the site $\delta(-)$ and then stay there forever. Analogous description for the site $L$. Moreover, all the Poisson clocks are taken as independent.

Recall \eqref{eq21} and \eqref{eq31}. 
Let us define 
\begin{equation}\label{dual}
F(n,\xi)=\frac{1}{\beta_+^{n_{\delta(+)}}\beta_-^{n_{\delta(-)}}}\prod_{x=-L}^L\frac{\xi_x^{n_x}}{n_x!}\,.
\end{equation}
\begin{theorem}[Duality]
Fix $\zeta\in \Omega_L$ and
\begin{equation}\label{eq32}
k=(0,k_{-L},\ldots,k_L,0)\in \bb N^{\Lambda_L}\,.
\end{equation}
Denote by $\bb E_k$ the expectation induced by the Markov process of generator $A_L$ starting at the configuration $k$ and denote by ${\bf E}_{\zeta}$ the expectation induced by the Markov process of generator $G_L$ starting at the initial configuration $\zeta$. Then, for all $t\geq 0$,
\begin{equation}\label{eq33}
{\bf E}_\zeta \Big[ F(k,\xi_t)\Big]\;=\;  
\bb E_k \Big[ F(\eta_t,\zeta)\Big].
\end{equation}
\end{theorem}
The proof of above is very similar to the one in \cite[Thm 2.1]{kmp}, and consists on checking that $A_LF=G_LF$. We left this to the reader. By letting $t\to\infty$ in \eqref{eq33} we obtain:
\begin{corollary}\label{cor32}
Given $k$ as in \eqref{eq31}, denote $\Vert k \Vert=\sum_{x=-L}^L k_x$. For $0\leq j\leq \Vert k \Vert$ let
\begin{equation*}
q_L(k,j)\;=\; {\bb P}_k\Big[\big\{ j \textrm{ particles hit } \delta(+) \textrm{ and } \Vert k \Vert-j \textrm{ particles hit } \delta(-)\big\}\Big].
\end{equation*}
Then
\begin{equation*}
\int F(k,\xi)\,\mu_L(d\xi)\;=\; \sum_{j=0}^{\Vert k \Vert}
\frac{1}{\beta_+^j\beta_-^{\Vert k \Vert-j}}\,q_L(k,j)\,.
\end{equation*}
\end{corollary}
\section{Label Process}\label{s4}
We will call by   \textit{Label Process} the Markov process constructed by labelling particles in the process presented in Definition \ref{def3} in the following way. First, we put a label to distinguish each particle. Consider a time when the Poisson clock associated to a pair of sites $k,k+1$ rings. In that moment, let us say that the total quantity of particles is $n_k+n_{k+1}$. 
Make a bijection between the set of labels of those particles and the set of integers $\{0,\ldots,n_k+n_{k+1}\}$. 
 Then, choose an integer $U$ uniformly between $0$ and $n_k+n_{k+1}$ and independently choose uniformly a permutation $\zeta$ of the integers $\{0,\ldots,n_k+n_{k+1}\}$. The particles corresponding to the first $U$ positions of the permutation $\zeta$ will be addressed to the site $k$ and the remaining to the site $k+1$. At the boundaries, the procedure is the following: when the Poisson clock associated to the site $L$ rings, all the particles in the site $L$ move to the site $\delta(+)$ and stay there forever. Analogous description for the site $-L$.
 
Notice that \textit{particles are not independent} in this dynamics. Moreover, by counting how many particles there is at each site we can recover the process given in the Definition \ref{def3}.

\begin{definition}
 Let $\bb P_{x_1,\ldots,x_n}^L$ be the probability induced by the Label Process starting from $n$ distinct particles located at the sites $x_1,\ldots,x_n\in \{-L,\ldots, L\}$.
Moreover, for $x_1,\ldots,x_n\in \{-L,\ldots, L\}$ and
$\eps_1,\ldots,\eps_n\in\{\,\delta(+)\,,\,\delta(-)\,\}$, denote
\begin{equation*}
\begin{split}
 p_L(x_1,\ldots,x_n;\eps_1,\ldots,\eps_n)
:= \bb P_{x_1,\ldots,x_n}^L\Big[\big\{\textrm{for }i=1,\ldots,n, &\textrm{ the particle starting}\\
& \textrm{at } x_i \textrm{ hits the site } \eps_i\big\}\Big]\,.\\
\end{split}
\end{equation*}
\end{definition}
From the definition above, we get
\begin{equation*}
q_L(k,j)\;=\; \sum p_L(x_1,\ldots,x_n;\eps_1,\ldots,\eps_n)\,,
\end{equation*}
where the sum above is taken over all sequences $\eps_1,\ldots,\eps_n$ such that the cardinality of the set $\{i\;;\; \eps_i=\delta(+)\}$ is equal to $j$. 

The next proposition tell us that if  some of the particles initially located  at $\{x_1,\ldots,x_n\}$ are removed (or some particles are added),  the behavior of the remaining particles is not modified. 
\begin{proposition}\label{prop33} Recall that $\bb P^L_{x_1,\ldots,x_n}$ denotes the probability induced by the Label Process starting from particles initially located at the sites $x_1\ldots,x_n\in\{-L,\ldots,L\}$.
Let $m<n$ and $\{i_1,\ldots,i_m\}\subset \{1,\ldots,n\}$. Define $y_k=x_{i_k}$. 
Then, the  probability  $\bb P^L_{x_1,\ldots,x_n}$ restricted to the class of events that depend only on the set of particles initially located at $\{x_{i_1},\ldots,x_{i_m}\}$ coincides with $\bb P^L_{y_1,\ldots,y_m}$.

In particular, for any $1\leq i\leq n$,
\begin{equation*}
\begin{split}
\sum_{\eps_i\in\{\delta(-),\delta(+)\}}&  p_L(x_1,\ldots,x_n;\eps_1,\ldots,\eps_n)\\
&\;=\;p_L(x_1,\ldots,x_{i-1},x_{i+1},\ldots,x_n;\eps_1,\ldots,\eps_{i-1},\eps_{i+1},\ldots,\eps_n)\,.\\
\end{split}
\end{equation*}
\end{proposition}

We point out that the original proof of this result presented in \cite[Prop 3.1]{kmp}, is of difficult reading\footnote{The proof of \cite[Prop 3.1]{kmp} is concerned with the case  $a=b=0$ and $A=B=1$ for our model. Anyway, since the boundaries do not play any role in the proof, the statements are essentially the same.}. For this reason, it is provided here an alternative proof based on the tricky combinatorial identity which we develop in the next lemma. Recall the convention that $\binom{n}{m}=0$ whenever $n,m$ are integers such that  $n<m$.
\begin{lemma}\label{lemma34}  For any $q\in\{0,\ldots,M\}$, 
\begin{equation}\label{eqA1}
\sum_{p=0}^N\binom{p}{q}\binom{N-p}{M-q}\;=\;\binom{N+1}{M+1}\,.
\end{equation}
\end{lemma}
\begin{proof}
We  prove \eqref{eqA1} by a combinatorial argument. That is, we are going to count, by two different procedures, how many ways we can choose a subset with $M+1$ objects from a set with $N+1$ objects. Obviously, a first answer  is  the right hand side of \eqref{eqA1}. Fix $q\in \{0,\ldots, M\}$.
Without loss of generality, let us say that the set of objects is  a set of real numbers $\mc O=\{a_1,\ldots,a_{N+1}\}$ such that  $a_i<a_j$ whenever $i<j$. Define
\begin{equation*}
\begin{split}
S_p\;:=\; \{A\subset \mc O\;;\;&|A|=M+1, \,a_{p+1}\in A, \,\textrm{ there are } q \textrm{ elements}\\
& \textrm{ in } A \textrm{ strictly} 
 \textrm{  smaller than } a_{p+1}, \textrm{ and there are }\\
 & M-q \textrm{ elements in } A \textrm{ strictly bigger than } a_{p+1}\}.\\
\end{split}
\end{equation*}
It is easy to check that $S_0,\ldots,S_{N}$ are disjoint sets, its union is the set of all subsets of $\mc O$ of cardinality $M+1$, and $|S_p|=\binom{p}{q}\binom{N-p}{M-q}$. This implies \eqref{eqA1}.
\end{proof}

\begin{proof}[Proof of Proposition \ref{prop33}]
We will make use of the graphical construction of particles systems via Poisson process. Assume the same Poisson process are used to evolve both the Label Process with particles starting from sites $x_1,\ldots,x_n$ and the Label Process with particles starting from $y_1,\ldots,y_m$. For short, we will call by  $x$-Label Process the first process and by  $y$-Label Process the second one.

If the Poisson clock associated to the site $\delta(+)$ rings, all the particles in the site $L$ moves to the site $\delta(+)$ and stay there forever, and analogous statement holds for the site $\delta(-)$. Both situations do not interfere in the trajectory of particles \textit{before} hitting $\{\delta(+),\delta(-)\}$. 
Therefore, consider the time when    the Poisson clock associated to a pair of sites $k,k+1$ rings, with $k\in\{-L,\ldots,L-1\}$. 

As described in the Label Process definition, let 
$n_k+n_{k+1}=N$ the total number of particles at the sites $k,k+1$ in that instant, being $M$ of those particles belonging to the $y$-Label Process. Fix a bijection between these set of particles and the set  of integers numbers $\{1,\ldots, N\}$. 
Let $U, \zeta$ are independent, being $U$ an uniform random variable in the set $\{1,\ldots,N\}$ and $\zeta$ is  uniformly chosen on the set  of permutations of $\{1,\ldots,N\}$. 
The particles corresponding the first $U$ positions of $\zeta$ will be sent to the site $k$ and the remaining particles will be sent to the site $k+1$. 

 Since an uniform permutation $\zeta$ on $N$ objects induces an uniform permutation on $M$ of them, we only need to assure that the quantity of particles belonging to the $y$-Label Process that will be sent to the site $k$  is  uniformly distributed on $\{0,1,\ldots,M\}$. Denote this quantity of particles by $Y$, which is a function of $U$ and $\zeta$. For short,  denote by $\bb P$ the probability from to the   space where  $U,\zeta$ have been constructed. In other words, our goal is to show  that
\begin{equation}\label{eq35}
\bb P\big[ \,Y=q\, \big]\;=\; \frac{1}{M+1}\,,
\end{equation} 
for any  $0\leq q\leq M$ integer.
We have that 
\begin{equation*}
\begin{split}
\bb P\big[ \,Y=q\, \big]&\;=\;\sum_{p=0}^N\bb P \big[\, Y=q\,\big|\,U=p\,\big]\cdot \bb P \big[\,U=p\,\big]\\
& \;=\;\frac{1}{N+1}\sum_{p=0}^N\bb P \big[\, Y=q\,\big|\,U=p\,\big] \;=\;\frac{1}{N+1}\sum_{p=0}^N\frac{\binom{p}{q}\binom{N-p}{M-q}}{\binom{N}{M}}\,,\\
\end{split}
\end{equation*}
where in the last equality we have used that $\zeta$ is picked uniformly at random in the set of permutations.
Recalling now Lemma \ref{lemma34} yields \eqref{eq35}, finishing the proof.
\end{proof}

\begin{proposition}\label{exchangeable}
Fix $N$ integer, $x_1,\ldots,x_N$ and $\eps_1,\ldots,\eps_N \in \{\delta(+),\delta(-)\}$. Then, for any $u\in(-1,1)$ and any permutation $\sigma$ of the set $\{1,\ldots, N\}$,  the following limit holds:
\begin{equation}\label{433}
\begin{split}
\lim_{L\to \infty} \Big\{ p_L(&x_1+\lfloor u L\rfloor, \ldots,x_N+\lfloor u L\rfloor;\eps_1,\ldots,\eps_N )\\
& -p_L(x_1+\lfloor u L\rfloor, \ldots,x_N +\lfloor u L\rfloor;\eps_{\sigma(1)},\ldots,\eps_{\sigma{(N)}} )\Big\} \;=\;0\,,\\
\end{split}
\end{equation}
where we have denoted by $\lfloor u L\rfloor$ the integer part of $uL$.
\end{proposition}
Since the boundaries do not play any role in the result above, its proof can be straightforwardly adapted from  \cite[Proposition 3.5]{galves1981} and for this reason we omit it here.
Next we calculate the hitting probabilities for a single particle. 
\begin{proposition}\label{prop44}
Consider the stopping times 
\begin{equation*}
\begin{split}
&\tau_{\delta(+)}\;=\; \inf\big\{ t>0\;;\; X_t=\delta(+)\big\}\; \textrm{ and } \;\tau_{\delta(-)}\;=\; \inf\big\{ t>0\;;\; X_t=\delta(-)\big\}\,.\\
\end{split}
\end{equation*}
Then, for any $x\in \{-L,\ldots,L\}$,
\begin{equation*}
p_L(x;\delta(+))\;=\;\bb P^L_x \Big[\tau_{\delta(+)} < \tau_{\delta(-)}\Big]\;=\;
\frac{AB(x+L) + BL^a}{2ABL  +AL^b+BL^a}\,.
\end{equation*}
In particular, the limit below, which we denote by $p(u)$, does not depend on $x_1$ regardless the chosen  values of $a$ and $b$:
\begin{align}
p(u)&\;:=\;\lim_{L\to\infty}p_L(x_1+\lfloor uL\rfloor,\delta(+))\;=\;\lim_{L\to\infty}p_L(\lfloor uL\rfloor,\delta(+)) \notag \\ 
& \;=\; \lim_{L\to \infty}\frac{AB(\lfloor uL\rfloor +L) + BL^a}{2ABL  +AL^b+BL^a}\,. \label{eq455}
\end{align}
\end{proposition}

\begin{proof}
In order to not overload notation, we write down
 $c=A/L^a$, $d=B/L^b$,  $a_{L+1}=\bb P^L_{\delta(+)} \Big[\tau_{\delta(+)} < \tau_{\delta(-)}\Big]$,  $a_{-(L+1)}=\bb P^L_{\delta(-)} \Big[\tau_{\delta(+)} < \tau_{\delta(-)}\Big]$ and $a_x=\bb P^L_x \Big[\tau_{\delta(+)} < \tau_{\delta(-)}\Big]$, for $x\in\{-L,\ldots,L\}$. 
 
 Applying the Markov Property, one can check that  $(a_x)_{x}$  is a solution of the following linear system: $a_{L+1}=1$, $a_{-(L+1)}=0$, $a_L=(a_{L-1}+d)/(1+d)$, $a_{-L}=a_{-L+1}/(1+c)$, and, for $-L+1\leq x\leq L-1$,  $a_x=(a_{x-1}+a_{x+1})/2$.

In general, it is not an easy task to exhibit the solution of a  system as above in a simple way. However, the fact that $a_x=(a_{x-1}+a_{x+1})/2$  holds for $-L+1\leq x\leq L-1$ leads us to the guess that, except at the boundaries, the solution should be a restriction  to the integers of some linear function.  With this guess in mind, one can  deduce by some long albeit elementary calculations that
\begin{equation*}
a_x\;=\; \begin{cases}
\;1, & \textrm{ if } x=L+1\,,\\
\displaystyle\frac{cdx+d(Lc+1)}{c(Ld+1)+d(Lc+1)}, & \textrm{ if } -L\leq x\leq L\,,\\
\;0, & \textrm{ if } x=-(L+1)\,,\\
\end{cases}
\end{equation*}
finishing the proof since $c=A/L^a$ and $d=B/L^b$.
\end{proof}

\begin{proposition}\label{prop45}
For any $x_1,\ldots,x_N$ integers, for any $\eps_1,\ldots,\eps_N \in \{\delta(+),\delta(-)\}$ and for any $u\in(-1,1)$, the following limit holds:
\begin{equation}\label{eqprod}
\begin{split}
\lim_{L\to \infty} \Big\{ p_L(x_1+\lfloor u L\rfloor, \ldots,x_N+\lfloor u L\rfloor&;\eps_i,\ldots,\eps_N )\\
&-\prod_{i=1}^{N} p_L(x_i+\lfloor u L\rfloor ; \eps_i)\Big\} \;=\;0.\\
\end{split}
\end{equation}
\end{proposition}

\begin{proof} Recall that $u\in (-1,1)$ is fixed. 
 For $N=2$, the result follows from the fact that, with high probability, particles starting at $x_1+\lfloor u L\rfloor$ and $x_2+\lfloor u L\rfloor$ will meet before hit  the boundaries. The  concerning technical details can be easily adapted from the proof of \cite[Lemma 3.5]{galves1981}. In possess of the case $N=2$, we proceed to prove the result for any positive integer.
Let $I_L=\{-L,\ldots, L\}$ and denote elements of $\{-1,+1\}^{I_L}$ by $\eta$. Given $\{x_1,\ldots,x_N\}\subset I_L$, we define
\[
\begin{split}
\alpha_L\big(\{\eta\in \{-1,+1\}^{I_L} &\;;\; \eta(x_i)=\eps_i \textrm{ for } i=1,\ldots, N\}\big)\\
&\;:=\; p_L(x_1+\lfloor u L\rfloor, \ldots,x_N+\lfloor u L\rfloor;\eps_i,\ldots,\eps_N )\,.\\
\end{split}
\]
Proposition \ref{prop33} guarantees that $\alpha_L$ is a  probability measure on  $\{-1,+1\}^{I_L}$. 

Now, we extend $\alpha_L$ to some probability measure $\tilde{\alpha}_L$ on $\{-1,+1\}^{\bb Z}$, being the particular choice of the extension not relevant. Consider the weak convergence of probability measures (see \cite{Bill}). Since $\{-1,+1\}^{\bb Z}$ is a compact space, by  Prohorov's Theorem there exists a probability measure  $\alpha$  which is a limit of $\tau_{\lfloor uL \rfloor}\tilde{\alpha}_L$ along some subsequence.

We claim now that $\alpha$ is  the only possible limit along subsequences of $\tau_{\lfloor uL \rfloor}\tilde{\alpha}_L$ and, moreover, it  is a Bernoulli product measure of constant parameter. Notice that this claim  put together with Proposition \ref{prop44} immediately  imply \eqref{eqprod}, finishing the proof.
 
Proposition \ref{exchangeable} implies that $\alpha$ is an exchangeable measure. Hence, by de Finetti's Theorem, we conclude that $\alpha$ is a mixture of Bernoulli product measures, that is,
\[
\alpha\;=\; \int_0^1 \Theta_p \,m(dp)\,,
\]
where $\Theta_p$ is the Bernoulli product measure on $\{-1,+1\}^{\bb Z}$ of constant parameter $p\in[0,1]$ and $m$ is a probability measure on the Borelian sets of $[0,1]$, called the \textit{law of the mixture}. On de Finetti's Theorem and exchangeability, we refer to the  survey \cite{kingman1978}.

Provided by the case $N=2$, we already know that the marginal of $\alpha$ at the sites $x,x+1\in \bb Z$ is a Bernoulli product measure with same parameter at $x$ and $x+1$, because of the Proposition \ref{prop44}. For this reason, one can deduce that
\[
\int_0^1 p^2 \,m(dp)\;=\; \Big(\int_0^1 p\,m(dp)\Big)^2.
\]
Since $f(p)=p^2$ is a strictly convex function, then $m$ must be a Delta of Dirac measure, which  implies that $\alpha$ is a Bernoulli product measure of constant parameter $p(u)$. This proves the claim and concludes the proof.
\end{proof}

We are in position to prove our main result. 
Recall the definitions of $\nu_T$ and  $\tau_{\lfloor uL \rfloor}\tilde{\mu}_L$.

\begin{proof}[Proof of Theorem \ref{thm21}] Fix $u\in(-1,1)$.
We start with two observations. First,  for an exponential law we have that
\begin{equation*}
\int_0^\infty\frac{y^j}{j!}\lambda e^{-\lambda y}\,dy\;=\; \lambda^{-j}\,,\quad \forall j\in \bb N\,.
\end{equation*}
Thus, since $\nu_T$ is a product measure, 
we have that $\int F(k,\xi)\,\nu_T(d\xi)\;=\; \big({\kk} T\big)^{\Vert k \Vert}
$, for any $k\in \bb N^{\bb Z}$ such that $\Vert k \Vert<\infty$. Second, the  class of polynomials  $p:\bb R^{\bb Z}\to\bb R$  in a finite number of variables is a weak convergence-determining class for the set of probability measures  concentrated on $\bb R_+^{\bb Z}$.

These two observations implies that, in order to prove that	$\tau_{\lfloor uL \rfloor}\tilde{\mu}_L$ converges weakly to the probability measure $\nu_T$, it is sufficient to assure that, for any  
$k\in \bb N^{\bb Z}$ such that $\Vert k \Vert<\infty$, the following limit holds:
\begin{equation}\label{eq43}
\lim_{L\to\infty} \int F(k,\xi)\;\tau_{\lfloor uL \rfloor}\tilde{\mu}_L(d\xi) = \big({\kk} T\big)^{\Vert k \Vert}\,,
\end{equation}		 
where $T=T(u)$ is to be achieved according to the chosen values of the parameters $a$ and $b$. By the Corollary \ref{cor32}, the limit in the left side of above is equal to
\begin{equation}\label{eq44}
\sum_{j=0}^{\Vert k \Vert}
\frac{1}{\beta_+^j\beta_-^{\Vert k \Vert-j}}\,\lim_{L\to \infty}q_L(k+\lfloor uL\rfloor,j)\,.
\end{equation}
By the Proposition \ref{prop44} and the Proposition \ref{prop45}, we have that 
\begin{equation*}
\lim_{L\to \infty}q_L(k+\lfloor uL\rfloor,j)\;=\;\binom{\Vert k \Vert}{j} \big(p(u)\big)^j \big(1-p(u)\big)^{\Vert k \Vert-j}.
\end{equation*}
We therefore conclude that expression \eqref{eq44} is equal to 
\begin{equation*}
\Big(\frac{p(u)}{\beta_+}+\frac{1-p(u)}{\beta_-}\Big)^{\Vert k \Vert}\;=\; 
\Big({\kk} \,\Big\{p(u)T_+ +(1-p(u))T_-\Big\}\Big)^{\Vert k\Vert}.
\end{equation*}
Then, evaluating \eqref{eq455} in each regime of the parameters $a,b\in \bb R$ implies \eqref{eq43} where $T=T(u)$ is the one in the statement of the Theorem \ref{thm21}, finishing the proof.
\end{proof}

\section{Further extensions}\label{sec5}
\subsection{A variant of the KMP model}
We considerer here the model as defined in Section 2.4 of \cite{Out}, which is a slight variation of the original KMP model of \cite{kmp}. To better link the models, we adopt in this section \textit{ipse literis} the  notation of \cite{Out}. 
Define $\Lambda_N=\{1,\ldots, N-1\}$ and by $\xi\in \bb R_+^{\Lambda_N}$ the energy configuration of oscillators, being $\xi_x$ its energy at site $x\in \Lambda_N$. Given $p\in [0,1]$, denote by $\xi^{(x,y),p}$ the configuration that moves a fraction $p$ of the total energy at sites $x,y\in \Lambda_N$, that is,
\begin{equation*}
\xi^{(x,y),p}\;=\;
\begin{cases}
\xi_z & \text{ if } z\neq x,y\,,\\
p(\xi_x+\xi_y) & \text{ if } z= x\,,\\
(1-p)(\xi_x+\xi_y) & \text{ if } z=y\,.\\
\end{cases}
\end{equation*}
The KMP-type model  defined in \cite{Out} is the Markov process whose generator is
$L_N\;:=\; \sum_{x=0}^{N-1}L_{x,x+1}$, where, for $f:\bb R_+^{\Lambda_N}\to\bb R$,
\begin{equation*}
\begin{split}
&\big(L_{0,1}f\big)(\xi)\;=\;\int_0^{\infty}\Bigg[\int_0^1 \big[f(\xi^{(0,1),p})-f(\xi)\big]\,dp\Bigg]\frac{e^{-\frac{\xi_0}{T_0}}}{T_0}\,d\xi_0\,,\\
&\big(L_{x,x+1}f\big)(\xi)\;=\;\int_0^1 \Big[f(\xi^{(x,x+1),p})-f(\xi)\Big]\,dp\,, \quad \text{for }x=1,\ldots, N-2\\
&\big(L_{N-1,N}f\big)(\xi)\;=\;\int_0^{\infty}\Bigg[\int_0^1 \big[f(\xi^{(N-1,N),p})-f(\xi)\big]\,dp\Bigg]\frac{e^{-\frac{\xi_N}{T_1}}}{T_1}\,d\xi_N\,.\\
\end{split}
\end{equation*} 
This dynamics can be explained as follows: for the bulk $x=1,\ldots,N-1$, the dynamics is exactly the same of the original KMP model. But at the boundaries, at Poisson times of parameter one (for the left and right boundaries, respectively), the energy at the ghost sites $0$ and $N$ are replaced according with exponentials of  parameters $1/T_0$ and $1/T_1$, and then, \textit{immediately after that,} the total energy at these sites and its neighbors  is uniformly redistributed.

An analogous result of Theorem \ref{thm21} can be demonstrated as follows. First of all, similarly to what has been done in this paper,  one can consider the slowed/accelerated boundary version of the aforementioned KMP-type process, which is the dynamics defined through the generator
\begin{equation*}
L_N\;:=\; \frac{A}{N^a}L_{0,1}+\frac{B}{N^b} L_{N-1,N} +\sum_{x=1}^{N-2}L_{x,x+1}\,,
\end{equation*}
where $A,B,a,b\in \bb R$. The next step is to reach duality. As natural, the  dual process of  above is a Markov process of random walkers much similar to that one of Section \ref{s3}, with a different behavior near the boundaries. At the  bond $1,2$ it is associated a Poisson clock of parameter $A/N^a$. When this clock rings, particles are uniformly redistributed between the sites $x=1$ and $x=2$ and then, \textit{immediately}, all the particles left in the site $x=1$ are moved to the site $x=0$ and stay there forever. An analogous statement stands for the bond $N-1,N$. 

Putting it formally, define $\Gamma_N=\{0,1,\ldots, N\}$ and denote by $n\in \big(\bb N\cup \{0\}\big)^{\Gamma_N}$ a configuration of particles. The dual process will be described by the generator
\[
A_L\;=\; \frac{A}{N^a}A_{0,1}+\frac{B}{N^b} A_{0,1} + \sum_{x=1}^{N-1}A_{x,x+1}\,,
\]
where, for $f:\big(\bb N\cup \{0\}\big)^{\Gamma_N}\to \bb R$, 
\begin{equation*}
\begin{split}
&\big(A_{0,1}f\big)(n)\;=\;\frac{1}{n_1+n_2+1}\sum_{q=0}^{n_1+n_2} \Big[f(n_0+q,0,n_1+n_2-q,n_3,\ldots, n_N)-f(n)\Big]\,,\\
&\big(A_{x,x+1}f\big)(n)\;=\; \frac{1}{n_x+n_{x+1}+1}\sum_{q=0}^{n_x+n_{x+1}}\Big[f(n^{(x,x+1),q})-f(n)\Big]\,, \, x=1,\ldots, N-2,\\
&\big(A_{N-1,N}f\big)(n)=\\
&\frac{1}{n_{N-2}+n_{N-1}+1}\sum_{q=0}^{n_{N-2}+n_{N-1}} \Big[f(n_0,\ldots,n_{N-2}+n_{N-1}-q,0, n_N+q)-f(n)\Big]\,,\\
\end{split}
\end{equation*} 
and
\begin{equation*}
n^{(x,y),q}\;=\;
\begin{cases}
n_z, & \text{ if } z\neq x,y\,,\\
q, & \text{ if } z= x\,,\\
n_x+n_y-q, & \text{ if } z=y\,.\\
\end{cases}
\end{equation*}
As can be  directly checked, these two processes are dual with respect  to the corresponding version of   duality function \eqref{dual}, namely 
\begin{equation*}
F(n,\xi)\;:=\;(T_0)^{n_{0}}(T_1)^{n_{N}}\prod_{x=1}^{N-1}\frac{\xi_x^{n_x}}{n_x!}\,.
\end{equation*}
The remaining arguments of this paper remains in force, and, \textit{mutatis mutandis}, the  statement of Theorem \ref{thm21} holds.

\subsection{Symmetric Inclusion Process}
Another interesting model to be studied is the \textit{Symmetric Inclusion Process} (see \cite{Carinci2013,OpokuRedig} for a definition). Since a duality has been proved (under general boundary rates) in \cite[Section 4.1]{Carinci2013} and an asymptotic independence of the dual process of walkers has been proved in \cite{OpokuRedig}, apart  of some minor technical issues to be verified, an analogous result of Theorem \ref{thm21} should stand for this model. However, since in \cite{OpokuRedig} the propagation of local equilibrium has been proved (leading to the hydrodynamic limit), an even more interesting question arises: does the hydrodynamic limit for an accelerated/slowed boundary version of the Inclusion Process  obey different boundary conditions related to Theorem \ref{thm21}?

In view of this discussion, we decided to left this subject (local equilibrium and propagation of local equilibrium for  an accelerated/slowed boundary version of the Inclusion Process) for a future work.

\subsection{Others processes}
The paper \cite{Carinci2013} presents different models with similar duality results of that corresponding to the KMP model, as the Brownian Energy Process, among others.
However, since an asymptotic independence of the dual process is not available, it is not clear if a result in the sense of Theorem~\ref{thm21} can be proved for these processes, except in the  case of independent random walks, for which the asymptotic independence is obvious.



\section*{Acknowledgements}

The author would like to thank the anonymous Referee for valuable comments, which  lead to the writing of  Section~\ref{sec5}.
Besides, the author was supported through a project Jovem Cientista-9922/2015, FAPESB-Brazil.

\bibliography{bibliografia}
\bibliographystyle{plain}

\end{document}